\documentclass[12pt,reqno]{amsart}
\usepackage{amsfonts,amssymb,hyperref,amsthm,enumerate,color,stmaryrd,color}
\usepackage{amssymb,hyperref,stmaryrd,enumerate,pgf,tikz,color}
\usepackage[left=2.6cm,right=2.6cm,top=3cm,bottom=3cm,bindingoffset=0cm]{geometry}
\newtheorem{thm}{Theorem}[section]
\newtheorem{lem}[thm]{Lemma}
\newtheorem{prop}[thm]{Proposition}
\newtheorem{cor}[thm]{Corollary}
\newtheorem*{nota}{Notation}
\def\B{\mathcal{B}}
\def\P{\mathsf{Pet}}
\def\F{\mathcal{F}}
\def\G{\Gamma}
\def\la{\langle}
\def\ra{\rangle}
\def\Z{\mathbb{Z}}
\DeclareMathOperator{\aut}{Aut}
\DeclareMathOperator{\cay}{Cay}
\DeclareMathOperator{\cosg}{Cos}
\DeclareMathOperator{\soc}{soc}
\DeclareMathOperator{\sym}{Sym}
\begin{document}
\title[On certain edge-transitive bicirculants of twice odd order]
{On certain edge-transitive bicirculants of twice odd order}
\author[I.~Kov\'acs]{Istv\'an~Kov\'acs$^{\, 1,2}$}
\address{I.~Kov\'acs 
\newline\indent
UP IAM, University of Primorska, Muzejski trg 2, SI-6000 Koper, Slovenia 
\newline\indent
UP FAMNIT, University of Primorska, Glagol\v jaska ulica 8, SI-6000 Koper, Slovenia}
\email{istvan.kovacs@upr.si}
\author[J.~Ruff]{J\'anos Ruff$^{\, 2}$}
\address{J.~Ruff 
\newline\indent
University of P\'ecs, Institute of Mathematics and Informatics, Ifj\'us\'ag \'utja 6, 
H-7624 P\'ecs, Hungary} 
\email{ruffjanos@gmail.com}
\thanks{$^1$~Partially supported by the Slovenian Research Agency (research program P1-0285, research projects N1-0062, J1-9108, J1-1695, J1-2451 and N1-0208).
\newline\indent
$^2$~Partially supported by the ARRS-NKFIH Slovenian-Hungarian Joint Research Project, 
grant no.~SNN 132625 (in Hungary) and N1-0140 (in Slovenia). }
\maketitle
\begin{abstract}
A graph admitting an automorphism with two orbits of the same length is called 
a bicirculant. Recently, Jajcay et al.~initiated the investigation of the 
edge-transitive bicirculants with the properties that one of the subgraphs induced 
by the latter orbits is a cycle and the valence is at least $6$~
(Electron.~J.~Combin., 2019). 
We show that the complement of the Petersen graph is the only such graph whose 
order is twice an odd number. 
\end{abstract}
\section{Introduction}\label{sec:intro}
All groups and graphs in this paper will be finite.
A graph admitting an automorphism with two orbits of the same length is called 
a {\em bicirculant}.  
The symmetry properties of bicirculants have attracted 
considerable attention (see, e.g., \cite{AHK,CZFZ,DGJ,KKMW,MMSF,MP,P,ZZ}). 
Recently, Jajcay et al.~\cite{JMSV} initiated the investigation of the edge-transitive bicirculants with the properties that one of the subgraphs induced 
by the latter orbits is a cycle and the valence is at least $6$. 
Motivated by this, we set the following notation. 

\begin{nota}
For a positive integer $d \ge 3$, denote by 
$\F(d)$ the family of regular graphs having valence $d$ and admitting an automorphism with two orbits of the same length such that one of the 
subgraphs induces by these orbits is a cycle.
\end{nota}

The graphs in the family $\F(3)$ are the well studied {\em generalised Petersen graphs},  
which were introduced by Watkins~\cite{Wa} in 1969. 
The graphs in $\F(4)$ were defined under the name   
{\em Rose Window graphs} by Wilson~\cite{Wi} and those in $\F(5)$ under the name 
{\em Taba\v{c}jn graphs} by Arroyo et al.~\cite{AHKOS}.  
The question which of these graphs are edge-transitive have been 
answered in \cite{FGW,KKM,AHKOS}. Moreover, the automorphism groups of all 
(not only the edge-transitive) graphs in the families $\F(d)$, $d=3,4,5$, are also 
known (see~\cite{FGW,KKM,DKM,KMMS}). 

Jajcay et al.~\cite{JMSV} focused primarily on the family $\F(6)$, they 
called the members of this family {\em Nest graphs} (see also~\cite{V}). 
Their main result was the classification of 
edge-transitive Nest graphs of girth $3$, the task to classify all  
edge-transitive Nest graphs was posed as \cite[Problem~1.2]{JMSV}.
Regarding the families $\F(d)$ with $d > 6$, the following questions were raised (see~\cite[Question~1.1]{JMSV}): 

\begin{enumerate}[{\rm (1)}]
\setlength{\itemsep}{0.25\baselineskip}
\item For which $d > 6$ does the family $\F(d)$ contain at least one edge-transitive graph? 
\item For which $d > 6$ does the family $\F(d)$ contain infinitely many 
edge-transitive graphs? 
\end{enumerate}

Jajcay et al.~\cite{JMSV} also carried out an exhaustive computer search for edge-transitive graphs of order at most $220$ and belonging to $\F(6)$, and also 
for edge-transitive graphs of order at most $100$ and belonging to the families 
$\F(d)$ with $7 \le d \le 10$. 
By the {\em order} of a graph we mean the number of its vertices.  
They obtained $66$ graphs in $\F(6)$ (see~\cite[Table~1]{JMSV}) and none 
in the families $\F(d)$, $7 \le d \le 10$. Among 
the $66$ graphs, only one has twice odd order, and this graph is 
the complement of the {\em Petersen graph}.  
Motivated by these observations, in this paper 
we focus on the edge-transitive graphs in the families $\F(d)$,  
$d \ge 6$, whose order is twice an odd number.  
We would like to remark that the remaining edge-transitive graphs 
in $\F(6)$ are dealt with in~\cite{K2}.
\medskip

Our main result is the following theorem.

\begin{thm}\label{main}
The family $\F(d)$ with $d > 6$ contains no edge-transitive graph of  
twice odd order. Furthermore, the complement of 
the Petersen graph is the only edge-transitive graph in family $\F(6)$ of twice odd 
order.
\end{thm}

The paper is organised as follows. 
Section~\ref{sec:known} contains the needed results form graph and group theory.  
The next two sections are devoted to the preparation for the proof of our main theorem. 
The main result in Section~\ref{sec:cond} is Lemma~\ref{cond}, which contains   
some necessary conditions for a graph in $\F(6)$ to be edge-transitive.  
In Section~\ref{sec:min} we analyse the 
blocks of imprimitivity for a group of automorphisms acting transitively on the edges of 
a graph from $\F(d)$. The core of our proof lies in this analysis, by which we rely on the 
classification of primitive permutations groups containing a semiregular cyclic subgroup with two orbits  (see \cite{M}), and the 
classification of arc-transitive circulants (see~\cite{K,Li,LXZ}). The proof of Theorem~\ref{main} is presented in Section~\ref{sec:proof}.
\section{Preliminaries}\label{sec:known}
\subsection{Graph theory} 
For a graph $\G$, let $V(\G)$, $E(\G)$, $A(\G)$ and $\aut(\G)$ denote   
its {\em vertex set}, {\em edge set}, {\em arc set} and {\em automorphism group}, respectively. The set of vertices adjacent with a given vertex $v$ is denoted by $\G(v)$. 

Let $G \leq \aut(\G)$ and let $v \in V(\G)$. The {\em stabiliser} of $v$ in 
$G$ is denoted by $G_v$ and the {\em orbit} of $v$ under $G$ by $v^G$.  
For a subset $B \subseteq V(\G)$, the {\em set-wise stabiliser} of $B$ in $G$ is denoted 
by $G_{\{B\}}$. 
If $G$ is transitive on $V(\G)$, then $\G$ is said to be {\em $G$-vertex-transitive},     
and $\G$ is simply called {\em vertex-transitive} when it is $\aut(\G)$-vertex-transitive. The  ({\em $G$-}){\em edge-} and ({\em $G$-}){\em arc-transitive} graphs are defined correspondingly. 

Let $\G$ be a $G$-vertex-transitive graph. 
A subset $B \subseteq V(\G)$ is called a {\em block} 
for $G$ (the term {\em block of imprimitivity} is also commonly used) if $B^g \cap B=\emptyset$ or $B^g=B$ holds for every $g \in G$. The block $B$ is 
{\em non-trivial} if $1 < |B| < |V(\G)|$, and it is {\em minimal} if 
it is non-trivial and no non-trivial block is contained properly in $B$. 
The {\em block system} induced by $B$ is the partition of $V(\G)$ consisting of 
the images $B^g$, where $g$ runs over $G$. A block system is called  
{\em normal} if it consists of the orbits of a normal subgroup of $G$. 

Let $\pi$ be an arbitrary partition of $V(\G)$. 
For a vertex $v \in V(\G)$, let $\pi(v)$ denote the class containing $v$. 
The {\em quotient graph} of $\G$ with respect to $\pi$, denoted by $\G/\pi$, is 
defined to have vertex set $\pi$, and edges 
$\{ \pi(u), \pi(v) \}$, where $\{u,v\} \in E(\G)$ such that $\pi(u) \ne \pi(v)$. If there exists a constant $r$ such that 
$$
\forall \{u,v\} \in E(\G):~ \pi(u) \ne \pi(v)~\text{and}~|\G(u) \cap \pi(v)|=r,
$$
then $\G$ is called an {\em $r$-cover} of $\G/\pi$ (our definition of an $r$-cover generalises the definition given in~\cite{DGJ}, where $\pi$ is also assumed to be a block system).
The term {\em cover} will also 
be used instead of $1$-cover. In the special case when $\pi$ is formed by the orbits of 
an intransitive normal subgroup $N \lhd \aut(\G)$, $\G/N$ will also be written for 
$\G/\pi$, and the term {\em normal $r$-cover} ({\em normal cover}, respectively) 
will also be used instead of $r$-cover (cover, respectively). 
The following properties are well-known.

\begin{prop}\label{cover}
Let $\G$ be a connected $G$-vertex- and $G$-edge-transitive graph, let $\B$ 
be a normal block system of $G$, and let $K$ be the kernel of the action of $G$ on 
$\B$. 
\begin{enumerate}[{\rm (1)}]
\item $\G$ is a normal $r$-cover of $\G/\B$, where 
$r=|\G(v) \cap B|$, $v$ is any vertex and $B$ is any block in $\B$ 
containing a neighbour of $v$.    
\item If $\G$ is a normal cover of $\G/\B$, then $\G$ and $\G/\B$ have the same valence,    
the kernel $K$ is regular on every block in $\B$, and $\G/\B$ is $G/K$-edge-transitive.  
\end{enumerate}
\end{prop}

Let $S \subset H$ be a subset of a group $H$ such 
that $1_H \notin S$, where $1_H$ denotes the identity element of $H$. 
The {\em Cayley digraph} 
$\cay(H,S)$ is defined to have vertex set $H$ and arcs $(h,sh)$, where $h \in H$ and 
$s \in S$. In the case when $S$ is inverse-closed, we regard $\cay(H,S)$ as an undirected graph and use the term {\em Cayley graph}.
It is a well-known observation (see Sabidussi~\cite{S}) that, if  
$\G$ is any graph, $v$ is any vertex, and $H \le \aut(\G)$ is a regular subgroup, then 
\begin{equation}\label{eq:cay}
\G \cong \cay(H,S),~\text{where}~S=\{x \in H : v^x  \in \G(v)\}.
\end{equation} 

The Cayley digraphs of cyclic groups are shortly called {\em circulants}. 
A recursive classification of finite arc-transitive circulants was obtained independently by Kov\'acs~\cite{K} and Li~\cite{Li}. The paper~\cite{K} also provides an explicit 
characterisation~(see \cite[Theorem~4]{K}), which was rediscovered recently by 
Li et a.~\cite{LXZ}. The characterisation presented below follows from the proof of 
\cite[Theorem~4]{K} or from \cite[Theorem~1.1]{LXZ}. In order to state this characterisation, some further definitions are 
introduced next.
 
For a finite group $H$, denote by $H^\#$ the set of all non-identity elements of $H$, 
and by $H_R$ the group of all {\em right multiplications} by the elements of $H$.
If $K \le H$, then let $[H:K]$ denote the set of {\em right $K$-cosets} in $H$. 
In the case when $K$ is a block for $\aut(\cay(H,S))$, the block system induced by $K$ is 
equal to $[H:K]$. Now, if $K \lhd H$ also holds, then the image of the action 
of $H_R$ on $[H:K]$ is regular, in particular, $\G/[H:K]$ becomes a Cayley 
graph of the group $H/K$. 

A Cayley graph $\cay(H,S)$ is called {\em normal} if 
$H_R \lhd \aut(\cay(H,S))$. Note that, if $\cay(H,S)$ is a normal 
arc-transitive Cayley graph, then $S$ is equal to an $A$-orbit for some subgroup $A \le \aut(H)$. 

\begin{thm}\label{K}
{\rm (\cite{K})} 
Let $\G=\cay(C,S)$ be a connected arc-transitive graph, where $C$ is a cyclic group 
of order $n$. Then one of the following holds.
\begin{enumerate}[{\rm (a)}]
\item $\G$ is the complete graph.
\item $\G$ is normal.
\item There exists a subgroup $1 < D < C$ such that $D$ is a block for $\aut(\G)$ and 
$\G/[C:D]$ is a connected arc-transitive circulant. Furthermore, $S$ is a union of $D$-cosets. 
\item There exist subgroups $1 < D, E < C$ such that both $D$ and $E$ are 
blocks for $\aut(\G)$, $C=D \times E$, $|D| > 3$ and  
$\gcd(|D|,|E|)=1$. Furthermore, $S=D^\# R$, where 
$R \subseteq E^\#$, $R$ is inverse-closed, $\cay(E,R)$ is connected and 
arc-transitive.
\end{enumerate}
\end{thm}

Besides the Petersen graph, two further small arc-transitive graphs will appear later. 
The {\em Clebsch graph} is obtained from the $4$ dimensional cube graph $Q_4$ by 
adding the edges connecting antipodal points; the {\em lattice graph} $L_2(4)$ is defined to have vertices the ordered pairs $(i,j)$, $1  \le i, j \le 4$, and two vertices are adjacent if and only if either their first or second coordinates are the same. 
The graph $L_2(4)$ is depicted in Figure~1. It can be easily checked that the mapping $\sigma : (i,j) \mapsto (j+1,i)$, where the addition is computed modulo $4$, 
is an automorphism of $L_2(4)$, $\sigma$ has two orbits of the same length, and one of the subgraphs induced by these orbits is a cycle. 
All these show that $L_2(4)$ is an example of an edge-transitive graph from the 
family $\F(6)$. 
 
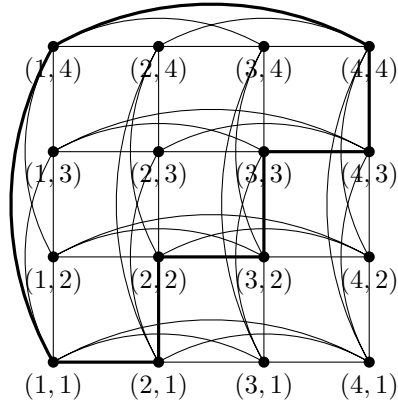
\begin{figure}[t!]
\begin{tikzpicture}[scale=0.7,line width=0.4pt]
\foreach \x in {0,2,4,6}
\draw[thin] (\x,0) -- (\x,6) (0,\x) -- (6,\x);
\foreach \x in {0,2,4} 
\draw[thin] 
(4,\x) arc (60:120:4)
(6,\x) arc (60:120:4)
(6,\x) arc (60:120:6);
\draw[thin]
(4,6) arc (60:120:4)
(6,6) arc (60:120:4);
\foreach \x in {2,4,6} 
\draw[thin] 
(\x,6) arc (150:210:6)
(\x,6) arc (150:210:4)
(\x,4) arc (150:210:4);
\draw[thin]
(0,6) arc (150:210:4)
(0,4) arc (150:210:4);
\draw[very thick] 
(0,0) -- (2,0) -- (2,2) -- (4,2) -- (4,4) -- (6,4) -- (6,6)
(6,6) arc (60:120:6)
(0,6) arc (150:210:6);
\foreach \x in {0,2,4,6}
\foreach \y in {0,2,4,6} 
\fill (\x,\y) circle (3pt);
\draw 
(0,0) node[below] {\footnotesize $(1,1)$}
(2,0) node[below] {\footnotesize $(2,1)$} 
(4,0) node[below] {\footnotesize $(3,1)$} 
(6,0) node[below] {\footnotesize $(4,1)$}
(0,2) node[below] {\footnotesize $(1,2)$}
(2,2) node[below] {\footnotesize $(2,2)$} 
(4,2) node[below] {\footnotesize $(3,2)$} 
(6,2) node[below] {\footnotesize $(4,2)$}
(0,4) node[below] {\footnotesize $(1,3)$}
(2,4) node[below] {\footnotesize $(2,3)$} 
(4,4) node[below] {\footnotesize $(3,3)$} 
(6,4) node[below] {\footnotesize $(4,3)$} 
(0,6) node[below] {\footnotesize $(1,4)$}
(2,6) node[below] {\footnotesize $(2,4)$} 
(4,6) node[below] {\footnotesize $(3,4)$} 
(6,6) node[below] {\footnotesize $(4,4)$}; 
\end{tikzpicture}
\caption{The lattice graph $L_2(4)$ and its subgraph induced by the orbit of the 
vertex $(1,1)$ under the automorphism 
$\sigma : (i,j) \mapsto (j+1,i)$, which is shown with thick lines.}
\end{figure}
\subsection{Group theory}
Our terminology and notation are standard and 
we follow the books~\cite{DM,H}. 
The {\em socle} of a group $G$, denoted by $\soc(G)$, is the subgroup generated by the set of all minimal normal subgroups (see~\cite[page~111]{DM}). 
The group $G$ is called {\em almost simple} if $\soc(G)=T$, where $T$ is a non-abelian simple 
group. In this case $G$ is embedded in $\aut(T)$ so that its socle is embedded 
via the inner automorphisms of $T$, and we also write $T \le G \le \aut(T)$. 

Our proof of Theorem~\ref{main} relies on 
the classification of primitive groups containing a cyclic subgroup with two orbits due to 
M\"uller~\cite{M}.  Here we need only the special case when the cyclic subgroup is semiregular. 

\begin{thm}\label{M}
{\rm (\cite[Theorem~3.3]{M})} 
Let $G$ be a primitive permutation group of degree $2n$ containing an element with two orbits of the same length. Then one of the following holds. 
\begin{enumerate}[{\rm (1)}] 
\item (Affine action) $G$ contains a regular normal subgroup isomorphic to $\Z_2^m$, 
where $m \in \{2,3,4\}$.\footnote{The group $G$ is from a short list, but 
as this possibility will not occur later, we omit the details.}
\item 
(Almost simple action) $G$ is an almost simple group and one of the following holds. 
\begin{enumerate}[{\rm (a)}]
\item $n \ge 3$, $\soc(G)=A_{2n}$, and 
$A_{2n} \le G \le S_{2n}$ in its natural action. 
\item $n=5$, $\soc(G)=A_5$, and $A_5 \le G \le S_5$ in its action on the set of $2$-subsets of $\{1,2,3,4,5\}$.
\item $n=(q^d-1)/2(q-1)$, $\soc(G)=PSL_d(q)$, and $PSL_d(q) \le G \le P\Gamma L_d(q)$ for some odd prime power $q$ and even number $d  \ge 2$ such that $(d,q) \ne (2,3)$.  
\item $n=6$ and $\soc(G)=G=M_{12}$.
\item $n=11$, $\soc(G)=M_{22}$, and $M_{22} \le G \le \aut(M_{22})$.
\item $n=12$ and $\soc(G)=G=M_{24}$.
\end{enumerate}
\end{enumerate}
\end{thm}

If $G$ is a group in one of the families (a)-(f) above, then it follows from 
\cite[Theorem~4.3B]{DM} that $\soc(G)$ is the unique 
minimal normal subgroup of $G$. Therefore, we have the following corollary.
 
\begin{cor}\label{M-p1}
Let $G$ be a primitive permutation group in one of the families (a)-(f) in part (2) of 
Theorem~\ref{M}, and let $N \lhd G$, $N \ne 1$. Then $N$ is also primitive.
\end{cor}

For a transitive permutation group $G \le \sym(\Omega)$, the {\em subdegrees} of 
$G$ are the lengths of the orbits of a point stabiliser $G_\omega$, 
$\omega \in \Omega$. 
Since $G$ is transitive, it follows that the subdegrees do not depend on the choice of 
$\omega$ (see~\cite[page~72]{DM}). The number of orbits of $G_\omega$ is called 
the {\em rank} of $G$. 
The actions of a group $G$ on sets $\Omega$ and $\Omega'$ are said to be 
{\em equivalent} if there is a bijection $\varphi : \Omega \to \Omega'$ 
such that 
$$
\forall \omega \in \Omega,~\forall g \in G:~
\varphi(\omega^g)=(\varphi(\omega))^g.
$$ 

Now, suppose that $G$ is a group in one of the 
families (a)-(f) in part (2) of Theorem~\ref{M}.
If $G$ is in family (a), then the action is unique up to equivalence and $G$ is clearly 
$2$-transitive. If $G$ is in family (b), then the action is unique up to equivalence and the subdegrees are $1, 3$ and $6$. 
Let $G$ be in family (c). The semiregular cyclic subgroup of $G$ with two orbits is 
contained in a regular cyclic group, called the 
{\em Singer subgroup} of $PGL_d(q)$ (see \cite[Chapter~2, Theorem~7.3]{H}). 
In this case the action is unique up to equivalence if and only if $d=2$. 
If $d \ge 4$, then the action of $G$ is equivalent to either its natural action on the set of 
points of the projective geometry $PG_d(q)$, or to its natural action on the set of hyperplanes of $PG_d(q)$. In both actions $G$ is $2$-transitive. 
Finally, if $G$ is in the families (d)-(f), then the action is unique up to equivalence  
and $G$ is $2$-transitive (this can also be read off from~\cite{Atlas}). 
All this information is summarised in the lemma below. 

\begin{lem}\label{M-p2&3}
Let $G$ be a primitive permutation group in one of the families 
(a)-(f) in part (2) of Theorem~\ref{M}.    
\begin{enumerate}[{\rm (1)}]
\item $G$ is $2$-transitive, unless $G$ belongs to family (b). In the latter case the  subdegrees are $1, 3$ and $6$. 
\item The action of $G$ is unique up to equivalence, unless $G$ is in family 
(c) and $d \ge 4$. In the latter case $G$ admits two inequivalent 
faithful actions, namely, the natural actions on the set of points and the set of 
hyperplanes, respectively, of the projective geometry $PG_d(q)$.
\end{enumerate}
\end{lem}

The following result about $G$-arc-transitive bicirculants was proved by Devillers et al.~\cite{DGJ}, but the proof works also for the edge-transitive bicirculants as well. 
In fact, it is an easy consequence of Theorem~\ref{M}. 

\begin{prop}\label{DGJ}
{\rm (\cite[part (1) of Proposition~4.2]{DGJ})}
Let $\G$ be a $G$-edge-transitive bicirculant such that $G$ is a primitive 
group. Then $\G$ is one of the following graphs:
\begin{enumerate}[{\rm (1)}] 
\item The complete graph, and $G$ is one of the $2$-transitive groups 
described in part (2) of Theorem~\ref{M}.
\item The Petersen graph or its complement, and $A_5 \le G \le S_5$.
\item The lattice graph $L_2(4)$ or its complement, and $G$ is a rank $3$ subgroup of 
$AGL(4,2)$.
\item The Clebsch graph or its complement, and $G$ is a rank $3$ subgroup of 
$AGL(4,2)$. 
\end{enumerate}
\end{prop}

One can easily check which of the graphs in the families (1)--(4) above 
belongs also to the family $\F(d)$ for some $d \ge 3$. 

\begin{cor}\label{prim}
Let $\G \in \F(d)$ be a $G$-edge-transitive graph for some $d \ge 3$. 
If $G$ is primitive on $V(\G)$, then $\G$ is isomorphic to $K_6$, or the Petersen graph, or its complement, or the lattice graph $L_2(4)$.  
\end{cor}
\section{A lemma on the graphs in the family $\F(6)$}\label{sec:cond}
In this section we derive some necessary conditions for 
a graph in $\F(6)$ to be edge-transitive. Our main tool is the coset graph construction defined next. 

Let $G$ be a group, let $H$ be a core-free subgroup of $G$, and let $S$ be a subset of $G$ such that $HSH=HS^{-1}H$. By {\em core-free} we mean that $H$ contains 
no non-trivial normal subgroup of $G$.  
The {\em coset graph} $\cosg(G,H,HSH)$ is defined to have vertex set $[G:H]$ (the set of all right $H$-cosets in $G$), and edges $\{Hx,Hy\}$, where $x, y \in G$  and $yx^{-1} \in HSH$. 
The lemma below is a folklore result. For the completeness of the paper, we provide 
the reader with a proof.

\begin{lem}\label{cosrepr}
Let $\G$ be a both $G$-vertex- and $G$-edge-transitive graph. 
Write $H$ for the vertex stabiliser $\aut(\G)_v$, and 
let $g \in G$ such that $\{v,v^g\}$ is an edge. Then the mapping  
$$
\varphi : [G:H] \to V(\G),~Hx \mapsto v^x~\text{for}~x \in G
$$
is an isomorphism between $\cosg(G,H,H\{g,g^{-1}\}H)$ and $\G$.
\end{lem}
\begin{proof}
We check first that $\varphi$ is indeed a mapping. 
Suppose that $Hx=Hy$ for some $x, y \in G$. Then $x=hy$ for some $h \in H$, 
and $\varphi(Hx)=v^x=v^{hy}=v^y=\varphi(Hy)$. This shows that $\varphi$ is 
well-defined. Since $\G$ is $G$-vertex-transitive, it follows that $\varphi$ is surjective. 
Using also that $|G|/|H|=|V(\G)|$, we find that $\varphi$ is a bijection. 

It remains to prove that 
$\{Hx,Hy\}$ is an edge of $\cos(G,H,H\{g,g^{-1}\}H)$ if and only 
if $\{v^x,v^y\}$ is an edge of $\G$.

Assume first that $\{Hx,Hy\}$ is an edge. By definition, $yx^{-1}=h_1gh_2$ or 
$h_1g^{-1}h_2$ for some $h_1, h_2 \in H$.  We deal only with the first case, the second 
one can be treated in the same way. Then $v^{yx^{-1}}=v^{h_1gh_2}=v^{gh_2}$, 
and as $\{v,v^{gh_2}\} \in E(\G)$, we find $\{v^x,v^y\} \in E(\G)$.

Now suppose that $\{v^x,v^y\} \in E(\G)$ for some $x, y \in G$. Since $\G$ is edge-transitive, there is some $g' \in G$, which maps the edge $\{v,v^{yx^{-1}}\}$ to 
$\{v,v^g\}$. This means that  
$$
v^{g'}=v,~v^{yx^{-1}g'}=v^g~\text{or}~v^{g'}=v^g,~v^{{yx^{-1}}g'}=v.
$$
In the first case $g' \in H$ and $yx^{-1}g'=hg$ for some $h\in H$, so 
$yx^{-1}=hg(g')^{-1} \in HgH$. In the second case $g'=h_1g$ and 
$yx^{-1}g'=h_2$ for some $h_1, h_2 \in H$, so $yx^{-1}=h_2(g')^{-1}=h_2g^{-1}h_1^{-1} \in 
Hg^{-1}H$. In either case we obtain that $\{Hx, Hy\}$ is an edge of the coset 
graph.   
\end{proof}

The valence of the graph $\cosg(G,H,H\{g,g^{-1}\}H)$ is given below.

\begin{lem}\label{LLM}
{\rm (\cite[Lemma~2.4]{LLM})}
The valence of $\cosg(G,H,H\{g,g^{-1}\}H)$ is equal to $|H|/|H \cap H^g|$ if 
$HgH=Hg^{-1}H$, or $2|H|/|H \cap H^g|$ otherwise.
\end{lem}

The main result of this section is the following lemma.

\begin{lem}\label{cond}
Let $\G$ be a $G$-edge-transitive graph in $\F(6)$ of order $2n$, and let 
$H$ be a vertex stabiliser of $G$. Then $G$ contains an element $g$ 
of order $n$ satisfying one of the following sets of conditions:
\begin{enumerate}[{\rm (1)}]
\item $HgH = Hg^{-1}H$ and   
$|H|=6|H \cap H^g|=\frac{1}{2}|H\la g \ra \cap HgH|$.
\item $HgH \ne Hg^{-1}H$ and  
$|H|=3|H \cap H^g|=|H\la g \ra \cap HgH|$.
\end{enumerate}
\end{lem}
\begin{proof}
There is a subgroup $C \le G$ and a vertex $v \in V(\G)$ such that 
$C$ is cyclic and semiregular with two orbits, and the subgraph of $\G$ induced by 
the orbit $v^C$ is a cycle. Choose $c \in C$ such that $\{v,v^c\}$ is an edge. 
Note that $c$ has order $n$.  

It is easy to see that $\G$ is not only $G$-edge- , but also $G$-vertex-transitive. 
Thus Lemma~\ref{cosrepr} can be applied to $G, K:=G_v$ and $c$, and this yields 
$$
\G \cong \G':=\cosg(G,K,K\{c,c^{-1}\}K).
$$

Note that, as both $H$ and $K$ are vertex stabilisers, 
$H=K^{g'}$ for some $g' \in G$. 

Assume first that $K c K=K c^{-1} K$.  
Using the formula for the valence of $\G'$ given in  
Lemma~\ref{LLM}, we find  
\begin{equation}\label{eq:val1}
|K|=6 |K \cap  K^c|.
\end{equation}

Let $\varphi$ be the isomorphism between $\G'$ and $\G$ 
defined in Lemma~\ref{cosrepr}.  
Since $K c K=K c^{-1} K$, it follows that 
$\G'$ is $G$-arc-transitive. This means that $\G$ is also $G$-arc-transitive, hence 
the stabiliser $K$ is transitive on $\G(v)$. It follows that 
$v^C \cap (v^c)^K=\{v^c,v^{c^{-1}}\}$. Applying $\varphi^{-1}$, we get 
$$
\{ Kc^i : 0 \le i \le n-1 \} \cap \{ K ck : k \in K \}=\{ Kc, Kc^{-1} \}.
$$ 
This shows that $K \la c \ra \cap K c K=K c \cup K c^{-1}$ holds in $G$, and so 
\begin{equation}\label{eq:cap1}
|K \la c \ra \cap 
K c K|=2|K|. 
\end{equation}
Using also that $H=K^{g'}$, \eqref{eq:val1} and 
\eqref{eq:cap1} show that choosing $g$ to be $c^{g'}$, part (1) 
of the lemma holds.

Now assume that $K c K \ne K c^{-1} K$.  
Using again Lemma~\ref{LLM}, we find  
\begin{equation}\label{eq:val2}
|K|=3 |K \cap  K^c|.
\end{equation}

In this case $\G'$ is not arc-transitive. Thus neither is $\G$, and $\G(v)$ splits into two $K$-orbits of the same size. We claim that $v^c$ and $v^{c^{-1}}$ belong to different 
$K$-orbits. For otherwise, $v^{c^{-1}}=v^{cg''}$ for some $g'' \in K$, and this would imply that the automorphism $cg''$ inverts the arc $(v^{c^{-1}},v)$, 
contradicting that $\G$ is not arc-transitive.  Then  
$v^C \cap (v^c)^K=\{v^c\}$, and this yields       
$K \la c \ra \cap K c K=K c$, and so  
\begin{equation}\label{eq:cap2}
|K \la c \ra \cap K c K|=|K|. 
\end{equation}
Then \eqref{eq:val2} and \eqref{eq:cap2} show that part (2) of the lemma holds
for $g=c^{g'}$.
\end{proof}
\section{Blocks}\label{sec:min}
Throughout this section we keep the following notation:
\medskip

\begin{quote}
$\G \in \F(d)$ is a $G$-edge-transitive graph of order $2n$  
for some $d \ge 6$.
\\ [+1ex] 
$C \le G$ is a cyclic semiregular subgroup with two orbits and one of the subgraphs  induced by these orbits is a cycle. 
\\ [+1ex]
$B$ is a non-trivial block for $G$ and $\B$ is the block system induced by $B$.
\end{quote}
\medskip

We say that $B$ is {\em cyclic} when 
it is contained in one of the $C$-orbits, and {\em non-cyclic} otherwise. 

Recall that, $C_{\{B\}}$ is 
the set-wise stabiliser of $B$ in $C$, and   
$\B$ is said to be normal when there is a normal subgroup $N$ of $G$ such that 
$\B$ consists of the $N$-orbits. 
   
\begin{lem}\label{b1}
Suppose that $B$ is cyclic such that $|B| < n/2$.
Then the kernel of the action of $G$ on $\B$ is equal to $C_{\{B\}}$. Furthermore, 
$\G$ is a normal cover $\G/\B$ and $\G/\B \in \F(d)$. 
\end{lem}
\begin{proof}
Denote by $V_i$ the $C$-orbits, $i=1,2$, and by 
$K$ the kernel of the action of $G$ on $\B$. 
It is clear that any block in $\B$ is contained in either $V_1$ or $V_2$. 
Consider the blocks contained in $V_1$. These form a block system for $C$, 
and as $C$ is regular on $V_1$, it follows that these blocks are the 
$C_{\{B_1\}}$-orbits, where $B_1$ is any block contained in $V_1$.
The group $C_{\{B_1\}}$ is regular on $B_1$, hence $|C_{\{B_1\}}|=|B_1|=|B|$, 
from which it follows that $C_{\{B\}}=C_{\{B_1\}}$. 
The same applies to $V_2$, and we conclude that $\B$ consists of the 
$C_{\{B\}}$-orbits. Thus $K \ge C_{\{B\}}$, in particular, $\B$ is normal. 

It can be assumed w.l.o.g.~that the subgraph of $\G$ induced by $V_1$ is a cycle. 
Now, fix an edge $\{u,v\}$ such that $u, v \in V_1$. 
Since $|B| < n/2$, it follows that $\G(u) \cap B'=\{v\}$, where $B'$ is the block containing $v$. By Lemma~\ref{cover}(1), $\G$ is a normal cover of $\G/\B$.
Then part (2) of the same lemma shows that $K$ is regular on every block, 
and so we have $K=C_{\{B\}}$.

In order to see that $\G/\B$ belongs to $\F(d)$, one only needs to observe that 
$\G/\B$ has valence $d$, $C/C_{\{B\}}$ is semiregular with two 
orbits, the induced cycle of $\G$ on $V_1$ projects 
to an induced cycle of $\G/\B$, and $V_1$ projects to a $C/C_{\{B\}}$-orbit. 
\end{proof}

\begin{lem}\label{b2}
Suppose that $B$ is non-cyclic. 
\begin{enumerate}[{\rm (1)}]
\item $B$ is a union of two $C_{\{B\}}$-orbits. The group $C$ acts transitively   
on $\B$ with kernel equal to $C_{\{B\}}$. 
\item If $|B| > 2$ and $B$ is minimal, then $\B$ is normal.
\end{enumerate}
\end{lem}
\begin{proof}
(1): Since there are two $C$-orbits on $V(\G)$ of the same size and 
$B$ has a point in common with both, it follows that $B$ 
splits into two $C_{\{B\}}$-orbits, hence $|B|=2|C_{\{B\}}|$. 

Let $\bar{C}$ and $K$ be the image and the kernel, respectively, 
of the action of $C$ on $\B$. It is clear that $C$ acts transitively on $\B$. 
This shows that $\bar{C}$ is regular, hence $C_{\{B\}} \le K$, and we can write  
$$
|C|/|K|=|\bar{C}|=|\B|=2n/|B|=|C|/|C_{\{B\}}|.
$$ 
This shows that $|K|=|C_{\{B\}}|$ also holds, and so $K=C_{\{B\}}$. 

(2): For a subgroup $X \le G_{\{B\}}$, denote by $X^*$ the image of 
the action of $X$ on $B$.  Let $M$ be the kernel of the action of $G$ on $\B$. By the minimality of $B$, 
$(G_{\{B\}})^*$ is primitive. On the other hand, as $1 <  (C_{\{B\}})^* \le M^* \lhd (G_{\{B\}})^*$, $M^*$ is transitive on $B$. This shows that $\B$ is normal. 
\end{proof}

Form now on we focus on the case when $n$ is odd. 

\begin{lem}\label{not thin}
Suppose that $n$ is odd and $B$ is minimal non-cyclic block for $G$. 
Then $|B| > 2$. 
\end{lem}
\begin{proof}
Assume on the contrary that $|B|=2$. Write $B=\{u,v\}$. 
Then $u^C \ne v^C$, and we may assume that 
the subgraph of $\G$ induced by $u^C$ is a cycle. Let $c \in C$ such that 
$\{u,u^c\}$ is an edge. Clearly, $c$ has order $n$.
  
There is a unique number $1 \le k \le (n-1)/2$ such that $v^{c^k}$ and 
$v^{c^{-k}}$ are the neighbours of $v$. Define the subset $S \subseteq C$ as 
$$
S=\{ x \in C : \{u,v^x\} \in E(\G) \}.
$$
It is clear that $|S|=d-2$. 
Also, $1_C \notin S$, where $1_C$ is identity element of $C$. 
For otherwise, $\{u,v\}$ is an edge, but as it is also a block, $G_u=G_v$, 
and this contradicts that $\G$ is edge-transitive. 

We say that two blocks in $\B$ are adjacent when these are adjacent as vertices of 
$\G/\B$. It can be easily seen that any two subgraphs of $\G$ induced by 
the union of two adjacent blocks 
are isomorphic to the same graph, say $\Delta$. We claim that 
$$
\Delta \cong K_2~\text{or}~2 K_2.
$$ 
Assume for the moment that there exists some $s \in S$ such that 
$s \notin \{c,c^{-1},c^k,c^{-k}\}$. 
Then the subgraph induced by $\{u,v,u^s,v^s\} \cong 2K_2$ or 
$K_2$ depending whether $s^{-1} \in S$ or not, and the claim follows. 
Now, as $|S|=d-2 \ge 4$, we are left with the case when $k \ne 1$ and 
$S=\{c,c^{-1},c^k,c^{-k}\}$. In this case the subgraph induced by 
$\{u,u^c,v,v^c\}$ is the $3$-path $(v,u^c,u,v^c)$. 
Since $\G$ is $G$-edge-transitive, there is some $g \in G$ mapping 
$\{u,u^c\}$ to $\{u,v^c\}$. This implies that $g$ maps the $3$-path to itself, 
hence it induces an automorphism of it. This is clearly impossible, and so 
the claim is proved.

Moreover, the argument above also shows that we have the following options: 
\begin{equation}\label{eq:or}
(k=1~\text{and}~S=S^{-1})~\text{or}~(k \ne 1~\text{and}~S \cap S^{-1}=
\emptyset).
\end{equation}

Now, define the permutation $t$ of $V(\G)$ as  
$$
t=(u\, v)(u^c\, v^c) \cdots (u^{c^{n-1}}\, v^{c^{n-1}}).
$$ 
Observe that $t$ commutes with any element of $G$. In particular, 
$\hat{C}:=\la c, t \ra$ is a regular cyclic group.

Define next the graph $\G'$ by 
$$
V(\G')=V(\G)~\text{and}~ 
E(\G')=\{ \{u,u^c\}^g : g \in \la G, t \ra \}.
$$
Using that $\la G, t\ra=G \times \la t \ra$, we find 
$$
E(\G')=\{ \{u,u^c\}^g : g \in G\}~\cup~\{ \{v,v^c\}^g : g \in G\}.
$$
This can be used to find the neighbourhood $\G'(u)$. 
If $k=1$, then $E(\G')=E(\G)$, and so $\G'(u)=\G(u)$. 
If $k \ne 1$, then $E(\G')$ splits into two edge-orbits under $G$, and 
$t$ swaps these edge-orbits. This yields that  
$\G'(u)=\G(u) \cup \G(u^t)^t=\G(u) \cup \G(v)^t$.
Now, according to \eqref{eq:cay}, $\G' \cong \cay(\hat{C},S' \cup S'')$, where  
$$
S'=\{c, c^{-1}, c^k, c^{-k}\}~\text{and}~ 
S''=t S \cup t S^{-1}. 
$$

It follows from \eqref{eq:or} 
that the valence of $\G'$ is $2+|S|$ if $k=1$, and $4+2|S|$ if $k\ne 1$. 

By definition, $\G'$ is edge-transitive. It is well-known that it must be then 
arc-transitive as well, and therefore, $\G'$ belongs to one of the families (a)-(d) in  
Theorem~\ref{K}. We consider below all possibilities case by case. 

Family~(a):  $\G'$ is the complete graph. This contradicts that $\G'$ has even 
valence and order.

Family~(b): $\G'$ is normal. Then $S' \cup S''=c^A$ for some subgroup 
$A \le \aut(\hat{C})$. This contradicts that $c$ has order $n$, while 
$t s$ has even order for each $s \in S$. 

Family~(c): There exists a subgroup $1 < D < \hat{C}$ such that $S' \cup S''$ is a union of $D$-cosets. If $|D|$ is odd, then $D \le C$, and so 
$S'$ would be a union of $D$-cosets. This is clearly impossible. 
Let $|D|$ be even. Then $t \in D$, implying $t S'=S''$ and $t S''=S'$, and so 
$|S'|=|S''|$. This contradicts that $|S'|=2$ and $|S''|=|S| \ge 4$ if $k=1$, and 
$|S'|=4$ and $|S''|=2|S| \ge 8$ otherwise. 

Family~(d): There exist subgroups $1 < D, E < \hat{C}$ such that $\hat{C}=D \times E$, 
$|D| > 3$, $\gcd(|D|, |E|)=1$, and $S' \cup S''=D^\# R$ for some subset 
$R \subseteq E^\#$. 

Let $|D|$ be odd. Then $D \le C$.  
For every $i \in \{1,-1,k,-k\}$, $|D c^i \cap  S'|=|D|-1 \ge 3$.
It follows that $|D|=5$, $k \ne 1$, and 
$S' \subset D c$. This shows that  
$c^2 \in D$, whence $D=C$. On the other hand,  
$D$ is a block for $\aut(\G')$, and so $C$ is a block for 
$G$. This is impossible.

Let $|D|$ be even. Then $t \in D$ and $D$ can be written as 
$D=\la t \ra \times D'$. Also, 
$R \subset E \le C$. As $S' \cup S''=D^\# R$ is inverse-closed, 
so is $R$, in particular, $|R|$ is even.  
Also, $S'=D^\#R \cap C=(D^\# \cap C)R=(D')^\#R$. 
Thus $|S'|=(|D'|-1)|R|$, and these imply in turn that 
$|R|=2$, $|D'|=3$, $|S'|=4$, and $|S''|=|D^\#R|-|S'|=6$. 
We have seen above that this is 
impossible.
\end{proof}

Our last lemma is one of the crucial steps towards Theorem~\ref{main}.

\begin{lem}\label{cyclic}
If $n > 5$ is odd, then $G$ admits a non-trivial cyclic block.
\end{lem}
\begin{proof}
Since $n > 5$, it follows from Corollary~\ref{prim} that $G$ is imprimitive. 
Choose a minimal non-trivial block $B$ for $G$, denote by $\B$ the
block system induced by $B$, and let $K$ denote the kernel of the action of 
$G$ on $\B$.

We are done if $B$ is cyclic, hence we assume that $B$ is non-cyclic.  
By Lemma~\ref{not thin}, $|B| > 2$. 
As before, for a subgroup $X \le G_{\{B\}}$, $X^*$ denotes 
the image of the action of $X$ on $B$. 

Apply Lemma~\ref{b2} to $B$. This shows that $\B$ is normal 
and $(C_{\{B\}})^*$ is a cyclic semiregular subgroup of $(G_{\{B\}})^*$ with 
two orbits. As $B$ is minimal, $(G_{\{B\}})^*$ is also primitive, and therefore 
described by Theorem~\ref{M}. 
The fact that $n$ is odd shows that $(G_{\{B\}})^*$ is one of the groups in the families 
(a)-(f) in part (2) of Theorem~\ref{M}. Then $K^* \lhd (G_{\{B\}})^*$. By Corollary~\ref{M-p1}, $K^*$ is also primitive.  
We derive the lemma in three steps.
\medskip

\noindent
{\it Step~1.} $K$ acts faithfully on every block in $\B$.
\medskip

Assume on the contrary that $K$ acts unfaithfully on some block in $\B$.
Using the connectedness of $\G$, it is easy to show that there are adjacent 
blocks $B', B'' \in \B$ so that the kernel of the action of $K$ on $B'$ is non-trivial on $B''$. Denote by $N$ the latter kernel. Now, as $N \lhd K$ and 
$K$ is primitive on $B''$, $N$ is transitive on $B''$. This implies that any vertex 
in $B'$ is adjacent with any vertex in $B''$. This contradicts that the subgraph 
of $\G$ induced by one of the $C$-orbits is a cycle.
\medskip

Fix a vertex $u \in B$. 
\medskip

\noindent
{\it Step~2.} For each block $B' \in \B$ there exists a unique vertex $u' \in B'$ 
such that $K_u=K_{u'}$.
\medskip

Define the binary relation $\sim$ on $\B$ by letting $B' \sim B''$ if and only if 
the action of $K$ on $B'$ and $B''$, respectively, are equivalent. 
It is easy to show that $\sim$ is an equivalence relation. 

Let $B', B'' \in \B$ such that $B' \sim B''$ and let 
$g \in G$. We claim that $(B')^g \sim (B'')^g$.
There is a bijective mapping $\varphi$ from $B'$ to $B''$ such that 
$$
\forall v \in B',~\forall k \in K:~\varphi(v^k)=(\varphi(v))^k.
$$

Now, pick arbitrary $w \in (B')^g$ and $k \in K$. 
Let $\gamma_1$ be the bijection from $B'$ to $(B')^g$ defined by 
$\gamma_1(x)=x^g$ for each $x \in B'$, and let $\gamma_2$ be the bijection from 
$B''$ to $(B'')^g$ defined by 
$\gamma_2(x)=x^g$ for each $x \in B''$. 
We finish the proof of the claim by showing $\psi(w^k)=(\psi(w))^k$, 
where $\psi$ is the bijection defined as the composition $\psi=\gamma_2 \circ \varphi 
\circ \gamma_1^{-1}$.
Then $w=v^g$ for some $v \in B'$ and $gk=k'g$ for some $k' \in K$ 
because $K \lhd G$.  Thus  
$$ 
\psi(w^k)=\psi(v^{k'g})=(\varphi(v^{k'}))^g=\varphi(v)^{k'g}=(\varphi(v)^g)^k=
(\psi(w))^k.
$$

Thus $\sim$ is $G$-invariant, and therefore, it is a $G$-congruence.  
Due to \cite[Exercise~1.5.4]{DM}, the $\sim$-classes form a 
block system for $G$ with respect to its action on $\B$. 
Denote by $m$ the number of $\sim$-classes. 
As $|\B|$ is odd, so is $m$. 
On the other hand, by Lemma~\ref{M-p2&3}(2), 
$K$ has at most two inequivalent faithful actions, and we conclude that $m=1$.  

Let $B' \in \B'$ be an arbitrary block. Since $K$ acts equivalently on $B$ and $B'$, 
it follows by \cite[Lemma~1.6B]{DM} that there is an element $u' \in B'$ such that 
$K_u=K_{u'}$.  By Lemma~\ref{M-p2&3}(1), $K$ is $2$-transitive on $B'$, unless  
$|B'|=10$, $K=A_5$ or $S_5$, and it has subdegrees $1, 3$ and $6$. 
This shows that $K_x \ne K_{u'}$ for any vertex $x \in B'$ such that $x \ne u'$. 
On the other hand, $K_u=K_{u'}$ and this finishes off the proof of Step~2. 
\medskip

\noindent
{\it Step~3.} The vertices $u'$ defined in Step~2 form a cyclic block.
\medskip
  
Denote by $\hat{B}$ the set of all vertices $u'$ defined in Step~2.
The cardinality $|\hat{B}|=|\B|$, and as $|\B|$ is odd, we are done 
if we show that $\hat{B}$ is a block. 
Equivalently, $\hat{B}^g=\hat{B}$ or $B \cap \hat{B}=\emptyset$ holds for 
each $g \in G$. 

Suppose that $v^g \in \hat{B}$ for some $v \in \hat{B}$ and $g \in G$. 
We have to show that $\hat{B}^g=\hat{B}$. In fact, it is enough to show that 
$\hat{B}^g \subseteq \hat{B}$. 
Choose an arbitrary element $w \in \hat{B}$. 
Then $K_u=K_v=K_{v^g}=K_w$. Using also that $K \lhd G$, we can write
$$
K_{w^g}=K \cap G_{w^g}=K \cap (G_w)^g=
(K  \cap G_w)^g=(K_w)^g.
$$
The same argument shows that $K_{v^g}=(K_v)^g$, and thus  
$K_u=K_{v^g}=(K_v)^g=(K_w)^g=K_{w^g}$. By the definition of the set 
$\hat{B}$, $w^g \in \hat{B}$, and $\hat{B}^g \subseteq \hat{B}$ follows.  
\end{proof}
\section{Proof of Theorem~\ref{main}}\label{sec:proof}
Assume on the contrary that there is an edge-transitive graph 
$\G \in \F(d)$ of order $2n$ such that $d \ge 6$, $n$ is odd and $n > 5$. 
Choose $n$ to be the smallest possible, i.e., whenever $\G' \in \F(d')$ is 
edge-transitive of order $2n'$ such that $d' \ge 6$, $n' < n$ and $n'$ is odd, 
then $\G'$ is isomorphic to the complement of the Petersen graph. 
In what follows, we denote the latter graph by $\overline{\P}$. 

For the sake of simplicity, write $G$ for $\aut(\G)$. 
Let $C \le G$ be a cyclic semiregular subgroup with two orbits such that one of the subgraphs induced by these orbits is a cycle. 

It follows from Corollary~\ref{prim} that $G$ is imprimitive. Choose  
a minimal non-trivial block $B$ for $G$, denote by $\B$ be the block system 
induced by $B$, and let $K$ be 
the kernel of the action of $G$ on $\B$. 

Due to Lemma~\ref{cyclic} we may assume that $B$ is cyclic, i.e., any block 
in $\B$ is contained in one of the two $C$-orbits. 

As $n$ is odd, $|B| < n/2$. By Lemma~\ref{b1}, $C_{\{B\}} \lhd G$. 
Let $p$ be a prime divisor of $|B|$, and let $P \le C_{\{B\}}$ be the subgroup 
of order $p$. The group $P$ is characteristic in $C_{\{B\}}$ and 
as $C_{\{B\}} \lhd G$, it follows that $P \lhd G$. By Lemma~\ref{b1}, $\G$ is a normal cover of $\G/P$ and $\G/P \in \F(d)$. Due to Proposition~\ref{cover}(2), 
$\G/P$ is also $G/P$-edge-transitive. 
The order of $\G/P$ is $2n/p$, hence the minimality of $n$ yields 
$$
n=5p,~\G/P \cong \overline{\P},~\text{and}~G/P \cong A_5~\text{or}~S_5.
$$
The last condition follows from the fact that 
$G/P$ acts transitively on the edges of $\overline{\P}$.

If $p=3$ or $5$, then $\G$ has order $30$ or $50$ and its valence is $6$. 
It follows from \cite[Table~1]{JMSV} that no graph in $\F(6)$ of order $30$ or 
$50$ is edge-transitive. Thus $p > 5$, and the Zassenhaus theorem (see  
\cite[Chapter~1, Theorem~18.1]{H}) shows that there exists a subgroup $L < G$ such that  $G=P \rtimes L$, $P \cong \Z_p$ and $L \cong G/P \cong A_5$ or 
$S_5$. 

Fix $z$ to be a generator of $P$ and let $L$ be identified with $A_5$ or $S_5$. 
Note that every element of $G$ can be expressed as a product 
$$
z^i \lambda,~\text{where}~0\le i \le p-1~\text{and}~\lambda \in L.
$$

Let $N=C_L(P)$. Then $N \lhd L$ and $L/N$ is isomorphic to a subgroup of 
$\aut(P)$, in particular, it is a cyclic group. This 
implies that either $N=L$, or $L=S_5$ and $N=A_5$. 
Consequently, 
\begin{equation}\label{eq:G}
G=\begin{cases}
P \times L & \mbox{if}~N=L, \\
P \rtimes L & \mbox{if}~L=S_5~\text{and}~N=A_5.
\end{cases} 
\end{equation}
Furthermore, in the second case the action of $L$ on $P$ by conjugation is defined by 
\begin{equation}\label{eq:LonP}
(z^i)^\lambda=\begin{cases} 
z^i & \mbox{if}~\lambda~\text{is even}, \\ 
z^{-i} & \mbox{if}~\lambda~\text{is odd},
\end{cases} 
\end{equation}
where $0 \le i \le p-1$ and $\lambda \in L=S_5$.
\medskip

Let $H=N_L(\la (1,2,3) \ra)$. Then 
$$
H=\begin{cases} 
\big\la (1,2,3), (1,2)(4,5) \big\ra \cong S_3 & \mbox{if}~L \cong A_5, \\
\big\la (1,2,3), (1,2), (4,5) \big\ra \cong S_3 \times \Z_2 & 
\mbox{if}~ L \cong S_5.
\end{cases}
$$   

We show next that $H$ is a vertex stabiliser in $G$. 
First, as $\G$ is a normal cover of $\overline{\P}$, the vertex stabilisers in $G$ are  
isomorphic to the vertex stabilisers in $\bar{G}$, and therefore, 
they are isomorphic to $H$. It follows immediately from \eqref{eq:G} and 
\eqref{eq:LonP} that all the elements of order $3$ in 
$G$ are contained in $L$, and form a single conjugacy class within $G$. 
In particular, there exists a vertex stabiliser in $G$ containing $(1,2,3)$, 
let this vertex stabiliser be denoted by $M$. 
Clearly, $M \le N_G(\la (1,2,3) \ra)$. Using \eqref{eq:G} and \eqref{eq:LonP},  
we obtain
$$
N_G(\la (1,2,3) \ra)=\begin{cases}
P \times H & \mbox{if}~N=L, \\
P \rtimes H & \mbox{if}~L=S_5~\text{and}~N=A_5.
\end{cases}
$$
If $N_G(\la (1,2,3) \ra)=P \times H$, then it is clear that $H$ is its 
only subgroup of $N_G(\la (1,2,3) \ra)$ isomorphic to $H$, so 
$M=H$, i.e., $H$ is indeed a vertex stabiliser.

Let $N_G(\la (1,2,3) \ra)=P \rtimes H$ and suppose that $M \ne H$.
Then $z^i\lambda \in M$ for some $1 \le i \le p-1$ and $\lambda \in H$. 
If $\lambda$ is even, then $p$ divides the order of $z^i\lambda$, which is impossible 
because $M \cong H$. Thus $\lambda$ must be odd. Using that $|M|=|H|$, it 
is not hard to show that 
for any $\mu \in H$, there is an element $z' \in P$ such that $z' \mu \in M$. 
It follows from this that $(H \cap A_5)  \le M$. If $i=2j$, then $(z^i\lambda)^{z^j}=\lambda$, and we have 
$$
M^{z^j}=
\big\la H \cap A_5, z^i\lambda \big\ra^{z^j}=
\big\la (H \cap A_5)^{z_j}, (z^i\lambda)^{z_j} \big\ra=
\big\la H \cap A_5, \lambda \big\ra=H.
$$
If $i=2j+1$, then 
one finds in the same way that $M^{\lambda z^{(p-1)/2-j}}=H$. 
In either case we obtain that $H$ is a vertex stabiliser. 
\medskip

The desired contradiction will arise after applying Lemma~\ref{cond} to $H$. 
Due to this lemma, there is an element $g \in G$ of order $5p$ satisfying all conditions in 
either part (1) or 
(2) of Lemma~\ref{cond}. W.l.o.g.~we may write $g=z \sigma$, where 
$\sigma$ is a $5$-cycle in $S_5$.  
\medskip

\noindent
{\it Case 1.} $N=L$. 
\medskip

By \eqref{eq:G}, $G=P \times L$. 
Using that $z \lambda=\lambda z$ for every $\lambda \in H$, it is easy to see that 
$g^{-1} \notin H g H$. Thus $HgH \ne Hg^{-1}H$, and so we have 
$|H|=3 |H  \cap H^g|=|H\la g \ra \cap HgH|$.
Since $H^z=H$, it follows that $H^g=H^\sigma$. Also, 
$$
H\la g \ra=\bigcup_{i=0}^{p-1} z^i H\la \sigma \ra~\text{and}~
HgH=z H\sigma H.
$$
Thus the equalities $|H|=3 |H  \cap H^g|=|H\la g \ra \cap HgH|$ 
reduce to 
$$
|H|=3 |H  \cap H^\sigma| = |H\la \sigma \ra \cap H \sigma H|.
$$ 
A computation with the computer package {\sc Magma}~\cite{BCP} shows that 
no $5$-cycle $\sigma$ satisfies these conditions. 
\medskip

\noindent
{\it Case~2.} $L=S_5$ and $N=A_5$. 
\medskip

Let $H_1=H \cap A_5$. Note that $H=H_1 \cup (4,5)H_1$.
By \eqref{eq:G}, $G=K \rtimes L$, and the action of $L$ on $P$ by conjugation is 
described in \eqref{eq:LonP}.

Then $HgH=Hg^{-1}H$ if and only if $g^{-1} \in HgH$, and so  
$g^{-1}=z^{-1}\sigma^{-1}=\lambda_1 z\sigma \lambda_2$ for some $\lambda_1, \lambda_2 \in H$. 
It can be seen that both $\lambda_1$ and $\lambda_2$ must be odd. 
Then $\lambda_i=(4,5)\lambda_i'$ for some $\lambda_i' \in H_1$, where $i=1,2$, 
and it holds $\sigma^{-1}=(4,5)\lambda_1' \sigma (4,5)\lambda_2'$. 
A computation with {\sc Magma}~\cite{BCP} verifies that such 
$\lambda_1'$ and $\lambda_2'$ exist for any $5$-cycle $\sigma$. 
Thus $HgH=Hg^{-1}H$, and so we have 
$|H|=6 |H  \cap H^g| = \frac{1}{2} |H\la g \ra \cap HgH|$. Then 
$H^g=H_1^g \cup ((4,5)H_1)^g=H_1^\sigma \cup z^{-2}((4,5)H_1)^\sigma$, 
hence the first equality reduces to 
\begin{equation}\label{eq:1st}
|H|=6  |H \cap H_1^\sigma|.
\end{equation}

In order to rewrite the second equality, observe first that 
$$
H \la g \ra=\bigcup_{i=0}^{p-1} H z^i \la \sigma \ra=
\bigcup_{i=0}^{p-1} \big( z^i  H_1 \la \sigma \ra \cup z^{-i}(4,5)H_1 
\la \sigma \ra \big)=\bigcup_{i=0}^{p-1} z^i H \la \sigma \ra.
$$
On the other hand, $H g H=(H_1 \cup (4,5)H_1)z\sigma H= 
z H_1 \sigma H \cup z^{-1} (4,5)H_1 \sigma H$. Thus  
$|H|=\frac{1}{2} |H\la g \ra \cap HgH|$ reduces to 
\begin{equation}\label{eq:2nd}
|H \la \sigma \ra \cap H_1 \sigma H|+
|H \la \sigma \ra \cap (4,5)H_1 \sigma H|=2 |H|.
\end{equation}
A computation with {\sc Magma}~\cite{BCP} shows that no 
$5$-cycle $\sigma$ satisfies both \eqref{eq:1st} and \eqref{eq:2nd}. 
This completes the proof of Theorem~\ref{main}.

\end{document}